\documentclass{amsart}
\usepackage{amssymb}
\usepackage{amsfonts}

\newtheorem{theorem}{Theorem}
\theoremstyle{plain}

\newtheorem{corollary}{Corollary}

\newtheorem{definition}{Definition}

\newtheorem{lemma}{Lemma}

\numberwithin{equation}{section}

\begin{document}
\title{On new types of convex functions and their properties}
\author{M.Emin \"{o}zdemir }
\email{eminozdemir@uludag.edu.tr}
\curraddr{Bursa Ulu\i da\u{g} \"{U}niversity, Department of Mathematics\\
and Science Education , Bursa-T\"{u}rkiye}
\thanks{This paper is in final form and no version of it will be submitted
for publication elsewhere.}

\begin{abstract}
In this study, Firstly, we will write two new convex functions for $%
-1<n-\alpha \leq 1\ $and two new lemmas. Then we will find the relevance of
the two new lemmas to Caputo-left-sided derivatives under additional
conditions and draw conclusions about them. Secondly,We will give examples
that support it. In what follows, we will rewrite the Hermite Hadamard
integral inequality based on this new definitions. Furthermore, we will
write an elementary inequality valid for value $-1<n-\alpha \leq 1.$
\end{abstract}

\maketitle

\section{Introduction}

Many important inequalities are established for the class of convex
functions, but one of the most important is so-called Hermite--Hadamard's
inequality (or Hadamard's inequality). This double inequality is stated as
follows in literature:

Let$\ f:I\subseteq R\rightarrow R$ be a convex function and let$\ a,b\in I$,
with $a<b$. The following double inequality;%
\begin{equation}
f\left( \frac{a+b}{2}\right) \leq \frac{1}{b-a}\int\nolimits_{a}^{b}f\left(
x\right) dx\leq \frac{f\left( a\right) +f\left( b\right) }{2}  \label{a0}
\end{equation}

is known as Hermite- Hadamard's inequality for convex functions. For many
papers connected with $\left( \ref{a0}\right) $ for different types of
convex functions see \cite{A2}-\cite{A37} the references therein. The paper
is organized as follows :

Firstly, \ we need some representations and preliminary information that
appear in the inequalities we will write.

Caputo left-sided derivative%
\begin{equation*}
^{C}D_{^{a^{+}}}^{\alpha }\left[ f\right] (x)=\frac{1}{\Gamma \left(
n-\alpha \right) }\int_{a}^{x}\left( x-\xi \right) ^{n-\alpha -1}\frac{%
d^{n}\left( f\left( \xi \right) \right) }{d\xi ^{n}}dt,\;\;x>a
\end{equation*}

Caputo right-sided derivative

\begin{equation}
^{C}D_{^{b-}}^{\alpha }\left[ f\right] (x)=\frac{\left( -1\right) ^{n}}{%
\Gamma \left( n-\alpha \right) }\int_{x}^{b}\left( \xi -x\right) ^{n-\alpha
-1}\frac{d^{n}\left( f\left( \xi \right) \right) }{d\xi ^{n}}dt,\;\;x<b
\label{a19}
\end{equation}

As can be seen,\ In Caputo, The first calculated the derivative of the
integer order and then the integral of the non-integer order.

Second, we need the new definition and some of its implications, which we
present below.

\begin{definition}
$\left( \left( n-\alpha \ convex\right) \right) $\bigskip The function $%
f:I\subseteq \left[ 1,\infty \right) \rightarrow \mathbb{R}^{+}$ is said to
be $\left( n-\alpha \right) \ convex$-convex, if the following inequality
holds:%
\begin{equation}
\ f(tx+(1-t)y)\leq t^{n-\alpha }f(x)+\left\vert \left( t^{n-\alpha }-\left(
1-t\right) ^{n-\alpha }\right) \right\vert f(y)  \label{a1}
\end{equation}
\end{definition}

for $-1<n-\alpha \leq 1,\left( 0<n\leq 1,\ 0\leq \alpha <1\right) \ \ t\in
\left( 0,1\right] \ ,\ x,y\in I.\ $with $y<x.\ $If $(\ref{a1})$ reverses,
then $f$ is said to be concave on $I.$Below are some interpretations of the
above definition.\bigskip

\bigskip $a)\ \ \ $Since\ $~\left\vert t_{1}^{n-\alpha }-t_{2}^{n-\alpha
}\right\vert \leq \left\vert t_{1}-t_{2}\right\vert ^{n-\alpha },$ for $%
n-\alpha \in \left( -1,1\right] \ $and $\forall t_{1},t_{2}\in \lbrack 0,1]\
\ $the definition in$\left( \ref{a1}\right) $ can be rewritten as 
\begin{eqnarray*}
&&f(tx+(1-t)y) \\
&\leq &t^{n-\alpha }f(x)+\left\vert \left( t^{n-\alpha }-\left( 1-t\right)
^{n-\alpha }\right) \right\vert f(y)\  \\
&\leq &t^{n-\alpha }f(x)+\left\vert 2t-1\right\vert ^{n-\alpha }f(y).
\end{eqnarray*}

$b)\ $Since $n-\alpha \in \left( -1,1\right] \ ,\ $\ $\ $If we take into
account the fact $\left\vert 2t-1\right\vert ^{n-\alpha }\leq t^{n-\alpha }\
,\ $for $t\in \left[ \frac{1}{2},1\right] \ $we can rewrite the inequality $%
\left( \ref{a1}\right) .$

We obtain a new definitionwritten for convex functions: 
\begin{eqnarray*}
f(tx+(1-t)y) &\leq &t^{n-\alpha }f(x)+\left\vert \left( t^{n-\alpha }-\left(
1-t\right) ^{n-\alpha }\right) \right\vert f\left( y\right) \\
&\leq &t^{n-\alpha }f(x)+\left\vert 1-2t\right\vert ^{n-\alpha }f(y) \\
&\leq &t^{n-\alpha }f(x)+t^{n-\alpha }f(y) \\
&=&t^{n-\alpha }\left( f\left( x\right) +f\left( y\right) \right)
\end{eqnarray*}

From the first and last terms 
\begin{equation}
f(tx+(1-t)y)\leq t^{n-\alpha }\left( f\left( x\right) +f\left( y\right)
\right)  \label{a12}
\end{equation}

$c)\ \ $If we choose $t=\frac{1}{2}$\ , $n-\alpha =1\ $in $\left( \ref{a12}%
\right) \ $for $t\in \left[ \frac{1}{2},1\right] .\ $we obtain 
\begin{equation*}
f\left( \frac{x+y}{2}\right) \leq \frac{f\left( x\right) +f\left( y\right) }{%
2}
\end{equation*}

$d)\ \ $If we just choose $t=\frac{1}{2}\ $in $\left( \ref{a12}\right) $%
\begin{equation*}
f\left( \frac{x+y}{2}\right) \leq \left( \frac{1}{2}\right) ^{n-\alpha }%
\left[ f\left( x\right) +f\left( y\right) \right]
\end{equation*}

$e)\ \ \ $If we choose $t=\frac{\frac{1}{2}+1}{2}$\ in $\left( \ref{a12}%
\right) \ $for $t\in \left[ \frac{1}{2},1\right] \ $we have 
\begin{equation*}
f\left( \frac{3x+y}{4}\right) \leq \left( \frac{3}{4}\right) ^{n-\alpha }%
\left[ f\left( x\right) +f\left( y\right) \right]
\end{equation*}

$f)\ \ $If we choose $t=\frac{\frac{3}{4}+1}{2}$\ in $\left( \ref{a12}%
\right) \ $for $t\in \left[ \frac{1}{2},1\right] \ $we have%
\begin{equation*}
f\left( \frac{7x+6y}{8}\right) \leq \left( \frac{7}{8}\right) ^{n-\alpha }%
\left[ f\left( x\right) +f\left( y\right) \right]
\end{equation*}

$g)\ \ $If we choose and $t=\frac{\frac{7}{8}+1}{2}$\ in $\left( \ref{a12}%
\right) \ $for $t\in \left[ \frac{1}{2},1\right] \ $we have%
\begin{equation*}
f\left( \frac{15x+14y}{8}\right) \leq \left( \frac{15}{16}\right) ^{n-\alpha
}\left[ f\left( x\right) +f\left( y\right) \right]
\end{equation*}

\bigskip If we continue in this sense, we obtain the sequence $\left( \frac{%
2^{k}-1}{2^{k}}\right) ^{n-\alpha }$ which converges to $1\ $for $\left(
n-\alpha ;k\right) \rightarrow \left( \pm 1,\ \infty \right) $

\begin{equation*}
\left( \frac{1}{2}\right) ^{n-\alpha },\ \left( \frac{3}{4}\right)
^{n-\alpha },...,\left( \frac{2^{k}-1}{2^{k}}\right) ^{n-\alpha
},...\rightarrow 1\ (k=1,2,3,...)
\end{equation*}

That is, 
\begin{equation}
f\left( \frac{\left( \frac{2^{k}-1}{2^{k}}\right) x+\frac{1}{2^{k}}y}{2^{k}}%
\right) \leq \left( \frac{2^{k}-1}{2^{k}}\right) ^{n-\alpha }\left[ f\left(
x\right) +f\left( y\right) \right]  \label{a17}
\end{equation}

and without losing generality we can write 
\begin{equation}
f\left( \frac{\left( \frac{2^{k}-1}{2^{k}}\right) x+\frac{1}{2^{k}}y}{2^{k}}%
\right) \leq \left( \frac{2^{k}-1}{2^{k}}\right) ^{n-\alpha }f\left(
x\right) +\left( \frac{\left( 2^{k}\right) ^{2}-1}{2^{k}}\right) ^{n-\alpha
}f\left( y\right)  \label{a18}
\end{equation}

for $k=1,2,3,...\ $and \ $n-\alpha \in \left( -1,1\right] .$

$h)\ $ if we only integrate both sides of the inequality$\left( \ref{a12}%
\right) $\ according to $t\ $: 
\begin{equation*}
\frac{1}{x-y}\int\nolimits_{x}^{y}f\left( u\right) du\leq \frac{f\left(
x\right) +f\left( y\right) }{n-\alpha +1};\ \ u=tx+(1-t)y
\end{equation*}

$\imath )\ \ $If we choose $t=1$\ in $\left( \ref{a12}\right) \ $for $t\in %
\left[ \frac{1}{2},1\right] \ $we obtain%
\begin{equation*}
f\left( y\right) \geq 0
\end{equation*}

Additionally, If we take limit of $\left( \ref{a17}\right) $ we\ again find
the $f\left( y\right) \geq 0$ relation for $\left( n-\alpha ;k\right)
\rightarrow \left( \pm 1,\ \infty \right) $.

$i)\;$If we choose $n=\frac{1}{2}\ ,\alpha =-\frac{1}{2}\ $and $t=\frac{1}{2}
$\ in$\left( \ref{a1}\right) .\;$we obtain Jensen convexity:%
\begin{eqnarray*}
f(\frac{x+y}{2}) &\leq &t^{n-\alpha }f(x)+\left\vert \left( t^{n-\alpha
}-\left( 1-t\right) ^{n-\alpha }\right) \right\vert f(y)\  \\
&=&\frac{1}{2}f\left( x\right) +\left\vert \left( \left( \frac{1}{2}\right)
-\left( 1-\frac{1}{2}\right) \right) \right\vert f(y) \\
&=&\frac{1}{2}f\left( x\right) +\left\vert 0\right\vert .f(y)\leq \frac{1}{2}%
f\left( x\right) +\frac{1}{2}f\left( y\right) =\frac{f\left( x\right)
+f\left( y\right) }{2}
\end{eqnarray*}

We used the fact that inequality $\left\vert t_{1}^{n-\alpha
}-t_{2}^{n-\alpha }\right\vert \leq \left\vert t_{1}-t_{2}\right\vert
^{n-\alpha },$ for $n-\alpha \in \left( -1,1\right] ,$ and $\forall
t_{1},t_{2}\in \lbrack 0,1].$

Now, in the light of the informations above, let's rearrange inequality$%
\left( \ref{a18}\right) \ $as definition.

\begin{definition}
The function $f:I\subseteq \left[ 1,\infty \right) \rightarrow \mathbb{R}%
^{+} $ is said to be $k^{n-\alpha }\ -convex$, if the following inequality
holds:%
\begin{equation*}
f\left( \frac{\left( \frac{2^{k}-1}{2^{k}}\right) x+\frac{1}{2^{k}}y}{2^{k}}%
\right) \leq \left( \frac{2^{k}-1}{2^{k}}\right) ^{n-\alpha }f\left(
x\right) +\left( \frac{\left( 2^{k}\right) ^{2}-1}{2^{k}}\right) ^{n-\alpha
}f\left( y\right)
\end{equation*}
\end{definition}

for $k=1,2,3,...\ $and \ $n-\alpha \in \left( -1,1\right] .$

Example : The function 
\begin{equation*}
f\left( x\right) =x
\end{equation*}

is an $k^{n-\alpha }\ -convex.$ Because 
\begin{eqnarray*}
f\left( \frac{\left( \frac{2^{k}-1}{2^{k}}\right) x+\frac{1}{2^{k}}y}{2^{k}}%
\right) &=&\frac{\left( \frac{2^{k}-1}{2^{k}}\right) x+\frac{1}{2^{k}}y}{%
2^{k}} \\
&\leq &\left( \frac{2^{k}-1}{2^{k}}\right) ^{n-\alpha }x+\left( \frac{\left(
2^{k}\right) ^{2}-1}{2^{k}}\right) ^{n-\alpha }y \\
\ \ \ \ \ \ \ \ &\leq &\left( \frac{2^{k}-1}{2^{k}}\right) ^{n-\alpha
}f\left( x\right) +\left( \frac{\left( 2^{k}\right) ^{2}-1}{2^{k}}\right)
^{n-\alpha }f\left( y\right)
\end{eqnarray*}

In light of the above information, we write new theorems and new lemmas and
interpreted their results.

\section{Hermite-Hadamard Inequality for $\left( n-\protect\alpha \right)$
convex function}

\begin{theorem}
\bigskip \bigskip Let $\ f:I\subseteq \left[ 1,\infty \right) \rightarrow
R^{+}$ be an $\left( n-\alpha \right) $\ convex function. If $a<b$ and \ $%
f\in L[a,b]$, with $n-\alpha \in \left( -1,1\right] $ then the following
Hermite--Hadamard type inequality hold:%
\begin{equation}
2^{n-\alpha -1}f\left( \frac{a+b}{2}\right) \leq \frac{1}{b-a}%
\int\nolimits_{a}^{b}f\left( x\right) dx\leq \frac{f\left( a\right) +f\left(
b\right) }{n-\alpha +1}  \label{a2}
\end{equation}
\end{theorem}

\begin{proof}
firstly , let be $ta+(1-t)b=x,\ for\ t\in \left[ 0,1\right] $%
\begin{eqnarray*}
\frac{1}{b-a}\int\nolimits_{a}^{b}f\left( x\right) dx
&=&\int\nolimits_{0}^{1}f\left( ta+(1-t)b\right) dt \\
&\leq &\int\nolimits_{0}^{1}\left( t^{n-\alpha }f\left( a\right) +\left\vert
\left( t^{n-\alpha }-\left( 1-t\right) ^{n-\alpha }\right) \right\vert
f\left( b\right) \right) dt \\
&\leq &\int\nolimits_{0}^{1}\left( t^{n-\alpha }f\left( a\right) +\left\vert
1-2t\right\vert ^{n-\alpha }f\left( b\right) \right) dt \\
&=&\int\nolimits_{0}^{1}t^{n-\alpha }f\left( a\right)
dt+\int\nolimits_{0}^{1}\left\vert 1-2t\right\vert ^{n-\alpha }f\left(
b\right) dt \\
&=&f\left( a\right) \int\nolimits_{0}^{1}t^{n-\alpha }dt+f\left( b\right)
\int\nolimits_{0}^{\frac{1}{2}}\left( 1-2t\right) ^{n-\alpha }dt \\
&&+f\left( b\right) \int\nolimits_{\frac{1}{2}}^{1}\left( 2t-1\right)
^{n-\alpha }dt \\
&=&\frac{f\left( a\right) +f\left( b\right) }{n-\alpha +1}
\end{eqnarray*}%
From the first and last terms,we have%
\begin{equation}
\frac{1}{b-a}\int\nolimits_{a}^{b}f\left( x\right) dx\leq \frac{f\left(
a\right) +f\left( b\right) }{n-\alpha +1}  \label{a13}
\end{equation}%
Here we again used the fact that inequality $\left\vert t_{1}^{n-\alpha
}-t_{2}^{n-\alpha }\right\vert \leq \left\vert t_{1}-t_{2}\right\vert
^{n-\alpha },$ for $n-\alpha \in \left( -1,1\right] $ and $\forall
t_{1},t_{2}\in \lbrack 0,1],$\bigskip secondly,%
\begin{eqnarray*}
f\left( \frac{a+b}{2}\right) &=&f\left( \frac{\left[ ta+\left( 1-t\right) b%
\right] +\left[ \left( 1-t\right) a+tb\right] }{2}\right) \\
&=&f\left( \frac{1}{2}\left[ ta+\left( 1-t\right) b\right] +\frac{1}{2}\left[
\left( 1-t\right) a+tb\right] \right) \\
&\leq &\frac{1}{2^{n-\alpha }}f\left( ta+\left( 1-t\right) b\right) \\
&&+\left\vert \left( \left( \frac{1}{2}\right) ^{^{n-\alpha }}-\left( 1-%
\frac{1}{2}\right) ^{n-\alpha }\right) \right\vert f\left( \left( 1-t\right)
a+tb\right) \\
&=&\frac{1}{2^{n-\alpha }}f\left( ta+\left( 1-t\right) b\right) +0.f\left(
\left( 1-t\right) a+tb\right) \\
&\leq &\frac{1}{2^{n-\alpha }}f\left( ta+\left( 1-t\right) b\right) +\frac{1%
}{2^{n-\alpha }}f\left( \left( 1-t\right) a+tb\right)
\end{eqnarray*}%
From the first and last terms,we have%
\begin{equation}
f\left( \frac{a+b}{2}\right) \leq \frac{1}{2^{n-\alpha }}f\left( ta+\left(
1-t\right) b\right) +\frac{1}{2^{n-\alpha }}f\left( \left( 1-t\right)
a+tb\right)  \label{a14}
\end{equation}%
Now, If we integrate both sides of $\left( \ref{a14}\right) $ over $\left[
0,1\right] .\ $We\ have%
\begin{equation}
2^{n-\alpha -1}f\left( \frac{a+b}{2}\right) \leq \frac{1}{b-a}%
\int\nolimits_{a}^{b}f\left( x\right) dx  \label{a15}
\end{equation}%
Finally.$\ $If we combine inequalities $\left( \ref{a13}\right) $ and $\ref%
{a15}$ we get obtain we obtain\ $\left( \ref{a2}\right) $.

Here we used fact that $0<\frac{1}{2^{n-\alpha }}.$
\end{proof}

\begin{corollary}
1. \ If we choose $n$ and $\alpha $ so that $n=\frac{1}{k}\left(
k=1,2,3,4,...\right) ,\ \alpha =\frac{s-1}{s}\left( k,s=1,2,3,...\right) \;.$%
in $\left( \ref{a2}\right) \;$Then we have 
\begin{equation}
2^{\frac{1}{k}+\frac{1}{s}-2}f\left( \frac{a+b}{2}\right) \leq \frac{1}{b-a}%
\int\nolimits_{a}^{b}f\left( x\right) dx\leq \frac{f\left( a\right) +f\left(
b\right) }{\frac{1}{k}+\frac{1}{s}}  \label{a10}
\end{equation}
\end{corollary}

\ \ \ If we choose $k=1\ $and $s=1\ $and in\ $\left( \ref{a10}\right) \ $we
obtain directly well known Hermit- Hadamard inequalities in $\left( \ref{a0}%
\right) :$%
\begin{equation*}
f\left( \frac{a+b}{2}\right) \leq \frac{1}{b-a}\int\nolimits_{a}^{b}f\left(
x\right) dx\leq \frac{f\left( a\right) +f\left( b\right) }{2}
\end{equation*}

or If we take limit of the inequality $\left( \ref{a10}\right) $ for $\left(
n,\alpha \right) \rightarrow \left( 1,0\right) \ $we again obtain in $\left( %
\ref{a0}\right) .$

\ \ If we choose $n$ and $\alpha $ so that $n=\frac{1}{k}\ \left(
k=1,2,3,...\right) \;.$and $\ \alpha =\frac{1}{s}\left( s=2,3,...\right) $in 
$\left( \ref{a2}\right) \;$Then we have%
\begin{equation*}
2^{\frac{1}{k}-\frac{1}{s}-1}f\left( \frac{a+b}{2}\right) \leq \frac{1}{b-a}%
\int\nolimits_{a}^{b}f\left( x\right) dx\leq \frac{f\left( a\right) +f\left(
b\right) }{\frac{1}{k}-\frac{1}{s}+1}
\end{equation*}

On the other hand, Since

$\left( \frac{1}{k}-\frac{1}{s}-1,\frac{1}{k}-\frac{1}{s}+1\right) >\left(
-1,1\right) \ $for $\left( k,s\right) \rightarrow \left( \infty ,\infty
\right) \ $we obtain 
\begin{equation*}
f\left( \frac{a+b}{2}\right) \leq \frac{2}{b-a}\int\nolimits_{a}^{b}f\left(
x\right) dx\leq 2\left( f\left( a\right) +f\left( b\right) \right) 
\end{equation*}

\bigskip Suppose that the function $f:I\subseteq R\rightarrow \mathbb{R}^{+}$
is $\left( n-\alpha \right) \ $-convex. By applying the $\left( \ref{a2}\
\right) \ $on each of the intervals $\left[ a,\frac{a+b}{2}\right] \ $and $%
\left[ \frac{a+b}{2},b\right] \ $we get 
\begin{eqnarray*}
2^{n-\alpha -1}f\left( \frac{a+b}{2}\right)  &\leq &2^{n-\alpha -2}\left(
f\left( \frac{3a+b}{4}\right) +f\left( \frac{3b+a}{4}\right) \right)  \\
&\leq &\frac{1}{b-a}\int\nolimits_{a}^{b}f\left( x\right) dx \\
&\leq &\frac{1}{n-\alpha +1}\left\{ \frac{f\left( a\right) +f\left( b\right) 
}{2}+f\left( \frac{a+b}{2}\right) \right\}  \\
&\leq &\frac{1}{n-\alpha +1}\left\{ \frac{f\left( a\right) +f\left( b\right) 
}{2}\right\} 
\end{eqnarray*}%
If the division process is continued, the difference between the image of
the arithmetic mean of $\left[ a,b\right] $ and the arithmetic mean of $%
\left[ f\left( a\right) ,f\left( b\right) \right] $ will become smaller.$\ $

From first and last inequalities we obtain a Jensen type inequality for the $%
n-\alpha \ $convex functions : 
\begin{equation*}
2^{n-\alpha -1}f\left( \frac{a+b}{2}\right) \leq \frac{1}{n-\alpha +1}%
\left\{ \frac{f\left( a\right) +f\left( b\right) }{2}\right\}
\end{equation*}

\begin{theorem}
let be $f:I\subseteq \left[ 1,\infty \right) \rightarrow \mathbb{R}^{+}$ a $%
\ \left( n-\alpha \right) $ convex \ function with $n-\alpha \in \left( -1,1%
\right] ,\ t\in \left( 0,1\right] \ ,\ x,y\in I.\ $Then the following
inequality holds :
\end{theorem}

\begin{eqnarray}
&&2^{n-\alpha -1}f\left( \frac{a+b}{2}\right)   \label{a4} \\
&\leq &\frac{f\left( a\right) }{^{n-\alpha +2}}+\frac{f\left( b\right) }{%
2\left( n-\alpha +1\right) }\left( 1-\left( i\right) ^{2\left( n-\alpha
+1\right) }\right) 
\end{eqnarray}

Here $i^{2}=-1\left( i\in C\ \text{is complex number}\right) $

\begin{proof}
If $\ t=1/2$ is taken in inequality $\left( \ref{a1}\right) ,$

\begin{eqnarray*}
f\left( \frac{a+b}{2}\right) &\leq &\frac{1}{2^{n-\alpha }}f\left( x\right)
+\left\vert \left( \left( \frac{1}{2}\right) ^{^{n-\alpha }}-\left( 1-\frac{1%
}{2}\right) ^{n-\alpha }\right) \right\vert f\left( y\right) \\
&=&\frac{1}{2^{n-\alpha }}f\left( x\right) +\left\vert \left( \left( \frac{1%
}{2}\right) ^{^{n-\alpha }}-\left( 1-\frac{1}{2}\right) ^{n-\alpha }\right)
\right\vert f\left( y\right) \\
&=&\frac{1}{2^{n-\alpha }}f\left( x\right) +\left\vert 0\right\vert f\left(
y\right) \\
&\leq &\frac{1}{2^{n-\alpha }}f\left( x\right) +\frac{1}{2^{n-\alpha }}%
f\left( y\right)
\end{eqnarray*}

From the first and last terms,we have%
\begin{equation*}
2^{n-\alpha }f\left( \frac{a+b}{2}\right) \leq f\left( x\right) +f\left(
y\right)
\end{equation*}
\end{proof}

Now, for all$\ n-\alpha \in \left( -1,1\right] \ $,$\;\frac{1}{2^{n-\alpha }}%
>0,\ $Now let $x=ta+\left( 1-t\right) b,\ y=\left( 1-t\right) a+tb.\ $Then

\begin{equation}
2^{n-\alpha }f\left( \frac{a+b}{2}\right) \leq f\left( ta+\left( 1-t\right)
b\right) +f\left( \left( 1-t\right) a+tb\right)  \label{a5}
\end{equation}

Multiplying both sides of $\left( \ref{a5}\right) $ by \ $t\ $and taking
integral over $\left[ 0,1\right] \ $with respect to $t\ $with the convexity
of $f$ : 
\begin{eqnarray*}
2^{n-\alpha -1}f\left( \frac{a+b}{2}\right) &\leq &\int\nolimits_{0}^{1}t\
f\left( ta+\left( 1-t\right) b\right) dt+\int\nolimits_{0}^{1}t\ f\left(
\left( 1-t\right) a+tb\right) dt \\
&=&\int\nolimits_{0}^{1}\left\{ 
\begin{array}{c}
t^{n-\alpha +1}f(a)+t\ \left\vert \left( t^{n-\alpha }-\left( 1-t\right)
^{n-\alpha }\right) \right\vert f\left( b\right) dt \\ 
+t\ \left\vert \left( \left( 1-t\right) ^{n-\alpha }-t^{n-\alpha }\right)
\right\vert f\left( b\right) dt%
\end{array}%
\right\} \\
&\leq &\int\nolimits_{0}^{1}\left\{ t^{n-\alpha +1}f(a)+t\ \left\vert
2t-1\right\vert ^{n-\alpha }f\left( b\right) dt\right\}
+\int\nolimits_{0}^{1}t\ \left\vert 1-2t\right\vert ^{n-\alpha }f\left(
b\right) dt \\
&=&\frac{f\left( a\right) }{^{n-\alpha +2}}+f\left( b\right) \left\{ 
\begin{array}{c}
\int\nolimits_{0}^{\frac{1}{2}}t\left( 2t-1\right) ^{n-\alpha
}dt+\int\nolimits_{\frac{1}{2}}^{1}t\left( 1-2t\right) ^{n-\alpha }dt \\ 
+\int\nolimits_{0}^{\frac{1}{2}}t\left( 1-2t\right) ^{n-\alpha
}dt+\int\nolimits_{\frac{1}{2}}^{1}t\left( 2t-1\right) ^{n-\alpha }dt%
\end{array}%
\right\} \\
&=&\frac{f\left( a\right) }{^{n-\alpha +2}}+f\left( b\right) \left\{ \frac{1%
}{2\left( n-\alpha +1\right) }-\frac{\left( -1\right) ^{n-\alpha +1}}{%
2\left( n-\alpha +1\right) }\right\} \\
&=&\frac{f\left( a\right) }{^{n-\alpha +2}}+\frac{f\left( b\right) }{2\left(
n-\alpha +1\right) }\left( 1-\left( -1\right) ^{n-\alpha +1}\right)
\end{eqnarray*}

which completes the proof.

Here we used integrals

\begin{eqnarray*}
\int\nolimits_{0}^{\frac{1}{2}}t\left( 2t-1\right) ^{n-\alpha }dt &=&-\frac{1%
}{4}\left( \frac{\left( -1\right) ^{n-\alpha +2}}{n-\alpha +2}+\frac{\left(
-1\right) ^{n-\alpha +1}}{n-\alpha +1}\right) \\
,\int\nolimits_{\frac{1}{2}}^{1}t\left( 1-2t\right) ^{n-\alpha }dt &=&-\frac{%
1}{4}\left( \frac{\left( -1\right) ^{n-\alpha +1}}{n-\alpha +1}-\frac{\left(
-1\right) ^{n-\alpha +2}}{n-\alpha +2}\right) \ 
\end{eqnarray*}

\begin{eqnarray*}
\int\nolimits_{0}^{\frac{1}{2}}t\left( 1-2t\right) ^{n-\alpha }dt &=&\frac{1%
}{4}\left( \frac{1}{n-\alpha ++1}-\frac{1}{n-\alpha +2}\right) \\
,\ \int\nolimits_{\frac{1}{2}}^{1}t\left( 2t-1\right) ^{n-\alpha }dt &=&%
\frac{1}{4}\left( \frac{1}{n-\alpha +2}+\frac{1}{n-\alpha +1}\right)
\end{eqnarray*}

\bigskip

$i)\ $\ If\ we\ take \ limit \ for$\ \left( n,\alpha \right) \rightarrow
\left( 1^{-},1^{-}\right) \ $or $\left( n,\alpha \right) \rightarrow \left(
0^{+},0^{+}\right) $in $\left( \ref{a4}\right) \ $we have 
\begin{equation*}
f\left( \frac{a+b}{2}\right) \leq f\left( a\right) +\frac{f\left( b\right) }{%
2}
\end{equation*}

$ii)\ \ $If\ we\ choose\ \ $n-\alpha =\frac{1}{2}\ $in$\left( \ref{a4}%
\right) \ $we have%
\begin{equation*}
f\left( \frac{a+b}{2}\right) \leq 2\sqrt{2}\left( \frac{f\left( a\right) }{%
^{5}}+\frac{f\left( b\right) }{3}\left( 1+e^{-\frac{\pi }{2}i}\right) \right)
\end{equation*}

where $i\in C\ $is a complex number,\ 

Now, we will write two new lemmas and new results based on them for $\left(
n-\alpha \right) \ $convex functions.

\begin{lemma}
Let $f:I\subseteq \left[ 1,\infty \right) \rightarrow R^{+}\ $be a
differentiable mapping on $(y,x)$ with $y<x$, $f^{(n)}>0.$ If $\ f^{(n)}\in
L[y,x]$, for $0<n\leq 1,\ 0\leq \alpha <1,\ t\in \left[ 0,1\right] $ . The
following equality for Caputo right-sided derivative holds:%
\begin{eqnarray*}
&&\left[ \frac{\left( -1\right) ^{n}}{\left( x-y\right) ^{2}}\Gamma \left(
2\right) -\frac{\left( -1\right) ^{n}\ }{\left( x-y\right) ^{\left( n-\alpha
\right) }}\Gamma \left( n-\alpha +2\right) \right] \left(
C_{D_{x^{-}}}^{\alpha }f\right) \left( u\right)  \\
&=&\int\nolimits_{0}^{1}t\left( 1-t^{n-\alpha }\right) f^{\left( n\right)
}\left( tx+\left( 1-t\right) y\right) dt
\end{eqnarray*}
\end{lemma}

Here $\left( C_{D_{x^{-}}}^{\alpha }f\right) \left( u\right) \ $is as in$%
\left( \ref{a19}\right) \ $and $\ \Gamma \left( r\right)
=\int\limits_{0}^{\infty }e^{-t}t^{r-1}dt,\ \left( r>0\right) \ $is the
Gamma function .

\begin{proof}
Since $\ -1<n-\alpha \leq 1,\ t\in \left[ 0,1\right] \ or\ n-\alpha \in
\left( 0,1\right] \ $\ with $y<x\ ,$ $f^{(n)}>0\ $we have Firstly , by
integrating and setting $u=tx+\left( 1-t\right) y$ 
\begin{eqnarray*}
&&\int\nolimits_{0}^{1}t\left( 1-t^{n-\alpha }\right) f^{\left( n\right)
}\left( tx+\left( 1-t\right) y\right) dt \\
&=&\int\nolimits_{0}^{1}\left( t-t^{n-\alpha +1}\right) f^{\left( n\right)
}\left( tx+\left( 1-t\right) y\right) dt\ \ \ \ ,\ \ \ \ \left( u=tx+\left(
1-t\right) y\right)  \\
&=&\frac{1}{x-y}\int\nolimits_{y}^{x}\left( \frac{u-y}{x-y}-\left( \frac{u-y%
}{x-y}\right) ^{n-\alpha +1}\right) f^{\left( n\right) }\left( u\right) du \\
&=&\frac{1}{\left( x-y\right) ^{2}}\int\nolimits_{y}^{x}\left( u-y\right)
f^{\left( n\right) }\left( u\right) du \\
&&-\frac{1}{\left( x-y\right) ^{n-\alpha +2}}\int\nolimits_{y}^{x}\left(
u-y\right) ^{n-\alpha +1}f^{\left( n\right) }\left( u\right) du \\
&=&\frac{\left( -1\right) ^{n}}{\left( x-y\right) ^{2}}\Gamma \left(
2\right) \left( C_{D_{x^{-}}^{\alpha }}f\right) \left( u\right)  \\
&&-\frac{\left( -1\right) ^{n}}{\left( x-y\right) ^{n-\alpha +2}}\Gamma
\left( n-\alpha +2\right) \left( C_{D_{x^{-}}^{\alpha }}f\right) \left(
u\right)  \\
&=&\frac{\left( -1\right) ^{n}}{\left( x-y\right) ^{2}}\left[ \Gamma \left(
2\right) -\frac{1}{\left( x-y\right) ^{n-\alpha }}\Gamma \left( n-\alpha
+2\right) \right] \left( C_{D_{x^{-}}^{\alpha }}f\right) \left( u\right) 
\end{eqnarray*}
\end{proof}

\begin{lemma}
Let $f:I\subseteq \left[ 1,\infty \right) \rightarrow R^{+}\ $be a
differentiable mapping on $(y,x)$with $y<x$. If $\ f^{(n)}\in L[y,x]$, for $%
0<n\leq 1,\ 0\leq \alpha <1,\ t\in \left[ 0,1\right] $ The following
equality for Caputo right-sided derivative holds:%
\begin{eqnarray}
&&\frac{\left( -1\right) ^{n}}{\left( x-y\right) ^{n-\alpha +1}}\left[
\Gamma \left( n-\alpha +1\right) -\frac{1}{x-y}\Gamma \left( n-\alpha
+1\right) \right] \left( C_{D_{x^{-}}^{\alpha }}f\right) \left( u\right)
\label{a7} \\
&=&\int\nolimits_{0}^{1}t^{n-\alpha }\left( 1-t\right) f^{\left( n\right)
}\left( tx+\left( 1-t\right) y\right) dt
\end{eqnarray}
\end{lemma}

\begin{proof}
\begin{eqnarray*}
&&\int\nolimits_{0}^{1}t^{n-\alpha }\left( 1-t\right) f^{\left( n\right)
}\left( tx+\left( 1-t\right) y\right) dt \\
&=&\int\nolimits_{0}^{1}\left( t^{n-\alpha }-t^{n-\alpha +1}\right)
f^{\left( n\right) }\left( tx+\left( 1-t\right) y\right) dt\ \ \ \ ,\ \ \ \
\left( u=tx+\left( 1-t\right) y\right) \\
&=&\frac{1}{x-y}\int\nolimits_{y}^{x}\left( \left( \frac{u-y}{x-y}\right)
^{n-\alpha }-\left( \frac{u-y}{x-y}\right) ^{n-\alpha +1}\right) f^{\left(
n\right) }\left( u\right) du
\end{eqnarray*}%
\begin{eqnarray*}
&=&\frac{1}{\left( x-y\right) ^{n-\alpha +1}}\int\nolimits_{y}^{x}\left(
u-y\right) ^{n-\alpha }f^{\left( n\right) }\left( u\right) du \\
&&-\frac{1}{\left( x-y\right) ^{n-\alpha +2}}\int\nolimits_{y}^{x}\left(
u-y\right) ^{n-\alpha +1}f^{\left( n\right) }\left( u\right) du \\
&=&\frac{\left( -1\right) ^{n}}{\left( x-y\right) ^{n-\alpha +1}}\Gamma
\left( n-\alpha +1\right) \left( C_{D_{x^{-}}^{\alpha }}f\right) \left(
u\right) \\
&&-\frac{\left( -1\right) ^{n}}{\left( x-y\right) ^{n-\alpha +2}}\Gamma
\left( n-\alpha +2\right) \left( C_{D_{x^{-}}^{\alpha }}f\right) \left(
u\right) \\
&=&\frac{\left( -1\right) ^{n}}{\left( x-y\right) ^{n-\alpha +1}}\left[
\Gamma \left( n-\alpha +1\right) -\frac{1}{x-y}\Gamma \left( n-\alpha
+1\right) \right] \left( C_{D_{x^{-}}^{\alpha }}f\right) \left( u\right)
\end{eqnarray*}
\end{proof}

\begin{corollary}
Since$\ 0<n\leq 1,\ 0\leq \alpha <1,\ t\in \left[ 0,1\right] \ \ for\
n-\alpha \in \left( 0,1\right] \ $we have 
\begin{eqnarray*}
t &\leq &t^{n-\alpha }\ \Longrightarrow \ 1-t^{n-\alpha }\leq 1-t \\
&&\ \ \ t\left( 1-t^{n-\alpha }\right) f^{\left( n\right) }\left( tx+\left(
1-t\right) y\right) \\
&\leq &t^{n-\alpha }\left( 1-t\right) f^{\left( n\right) }\left( tx+\left(
1-t\right) y\right) \ ,\ \ \ \ \ \ \ \ 
\end{eqnarray*}
\end{corollary}

or 
\begin{eqnarray}
&&\int\nolimits_{0}^{1}t\left( 1-t^{n-\alpha }\right) f^{\left( n\right)
}\left( tx+\left( 1-t\right) y\right) dt  \label{a16} \\
&\leq &\int\nolimits_{0}^{1}t^{n-\alpha }\left( 1-t\right) f^{\left(
n\right) }\left( tx+\left( 1-t\right) y\right) dt
\end{eqnarray}

$\ $From $\left( \ref{a16}\right) ,$that is 
\begin{eqnarray*}
&&\frac{\left( -1\right) ^{n}}{\left( x-y\right) ^{2}}\left[ \Gamma \left(
2\right) -\frac{1}{\left( x-y\right) ^{n-\alpha }}\Gamma \left( n-\alpha
+2\right) \right] \\
&\leq &\frac{\left( -1\right) ^{n}}{\left( x-y\right) ^{n-\alpha +1}}\left[
\Gamma \left( n-\alpha +1\right) -\frac{1}{x-y}\Gamma \left( n-\alpha
+1\right) \right]
\end{eqnarray*}

or 
\begin{eqnarray*}
&&\frac{1}{\left( x-y\right) ^{2}}\left[ 1-\frac{1}{\left( x-y\right)
^{n-\alpha }}\Gamma \left( n-\alpha +2\right) \right] \\
&\leq &\frac{1}{\left( x-y\right) ^{n-\alpha +1}}\left[ 1-\frac{1}{x-y}%
\right] \Gamma \left( n-\alpha +1\right)
\end{eqnarray*}%
$\ $

\begin{theorem}
\bigskip $\ 0\leq \left[ \alpha \right] <1,\ t\in \left[ 0,1\right] ,$ $%
n=[\alpha ]+1,\;$\ If \ $f:I\subseteq \left[ 1,\infty \right) \rightarrow 
\mathbb{R}^{+},\ I\subset \lbrack 0,\infty )$ ,be a differentiable function
on I such that $f\in AC^{1}L[y,x]$ $\ $with $y<x.$If $\ \left\vert
f\right\vert \ $is $\left( n-\alpha \right) \ $is convex, The following
inequalities holds: 
\begin{eqnarray*}
&&\left\vert \left[ \frac{\left( -1\right) ^{n}}{\left( x-y\right) ^{2}}%
\Gamma \left( 2\right) -\frac{\left( -1\right) ^{n}\ }{\left( x-y\right)
^{\left( n-\alpha \right) }}\Gamma \left( n-\alpha +2\right) \right] \left(
C_{D_{x^{-}}}^{\alpha }f\right) \left( u\right) \right\vert \\
&\leq &\left\vert f\left( x\right) \right\vert \left( \frac{1}{n-\alpha +2}-%
\frac{1}{2\left( n-\alpha +1\right) }\right)
\end{eqnarray*}%
\begin{equation}
+\frac{\left\vert f\left( y\right) \right\vert }{4}\left[ \beta \left(
n-\alpha +1,2\right) -\frac{\beta \left( n-\alpha +1,n-\alpha +2\right) }{%
2^{n-\alpha +2}}\right]  \label{a9}
\end{equation}%
\begin{equation*}
+\frac{\left\vert f\left( y\right) \right\vert }{4}\left( -1\right)
^{n-\alpha +1}\left[ 1-\frac{1}{2^{n-\alpha +1}}\right] \beta \left(
n-\alpha +1,n-\alpha +2\right)
\end{equation*}
\end{theorem}

\bigskip Here is $\beta \left( a,b\right) =\int\limits_{0}^{1}t^{x-1}\left(
1-t\right) ^{y-1}dt\ ,\ \left( x,y>0\right) \ \ $and$\ \ \Gamma \left(
r\right) =\int\limits_{0}^{\infty }e^{-t}t^{r-1}dt,\ \left( x>0\right) .$

Considering Lemma 1 and $\left( n-\alpha \right) \ $convexity of $\
\left\vert f\right\vert \ ,$ 
\begin{eqnarray*}
&&\left\vert \left[ \frac{\left( -1\right) ^{n}}{\left( x-y\right) ^{2}}%
\Gamma \left( 2\right) -\frac{\left( -1\right) ^{n}\ }{\left( x-y\right)
^{\left( n-\alpha \right) }}\Gamma \left( n-\alpha +2\right) \right] \left(
C_{D_{x^{-}}}^{\alpha }f\right) \left( u\right) \right\vert \\
&\leq &\left\vert \int\nolimits_{0}^{1}t\left( 1-t^{n-\alpha }\right)
f\left( tx+\left( 1-t\right) y\right) dt\right\vert
\end{eqnarray*}%
\begin{eqnarray*}
&&\left\vert \int\nolimits_{0}^{1}t\left( 1-t^{n-\alpha }\right) f\left(
tx+\left( 1-t\right) y\right) dt\right\vert \\
&\leq &\int\nolimits_{0}^{1}\left\vert t\left( 1-t^{n-\alpha }\right)
f\left( tx+\left( 1-t\right) y\right) \right\vert dt
\end{eqnarray*}%
\begin{eqnarray*}
&=&\int\nolimits_{0}^{1}t\left( 1-t^{n-\alpha }\right) \left\vert f\left(
tx+\left( 1-t\right) y\right) \right\vert dt \\
&\leq &\int\nolimits_{0}^{1}t\left( 1-t^{n-\alpha }\right) \left\{
t^{n-\alpha }\left\vert f\left( x\right) \right\vert +\left\vert
1-2t\right\vert ^{n-\alpha }\left\vert f\left( y\right) \right\vert \right\}
dt
\end{eqnarray*}%
\begin{equation*}
=\left\vert f\left( x\right) \right\vert \int\nolimits_{0}^{1}t^{n-\alpha
+1}\left( 1-t^{n-\alpha }\right) dt+\left\vert f\left( y\right) \right\vert
\int\nolimits_{0}^{1}t\left( 1-t^{n-\alpha }\right) \left\vert
1-2t\right\vert ^{n-\alpha }dt
\end{equation*}%
\begin{eqnarray*}
&&\left\vert f\left( x\right) \right\vert \int\nolimits_{0}^{1}t^{n-\alpha
+1}\left( 1-t^{n-\alpha }\right) dt+\left\vert f\left( y\right) \right\vert
\\
&&\times \left\{ \int\nolimits_{0}^{\frac{1}{2}}t\left( 1-t^{n-\alpha
}\right) \left( 1-2t\right) dt+\int\nolimits_{\frac{1}{2}}^{1}t\left(
1-t^{n-\alpha }\right) \left( 2t-1\right) dt\right\}
\end{eqnarray*}%
\begin{eqnarray*}
&=&\left\vert f\left( x\right) \right\vert \left( \frac{1}{n-\alpha +2}-%
\frac{1}{2\left( n-\alpha +1\right) }\right) \\
&&+\frac{\left\vert f\left( y\right) \right\vert }{4}\left[ \beta \left(
n-\alpha +1,2\right) -\frac{\beta \left( n-\alpha +1,n-\alpha +2\right) }{%
2^{n-\alpha +2}}\right]
\end{eqnarray*}%
\begin{equation*}
+\frac{\left\vert f\left( y\right) \right\vert }{4}\left( -1\right)
^{n-\alpha +1}\left[ 1-\frac{1}{2^{n-\alpha +1}}\right] \beta \left(
n-\alpha +1,n-\alpha +2\right)
\end{equation*}%
Here 
\begin{equation*}
\int\nolimits_{0}^{1}t^{n-\alpha +1}\left( 1-t^{n-\alpha }\right) dt=\frac{1%
}{n-\alpha +2}-\frac{1}{2\left( n-\alpha +1\right) }
\end{equation*}%
\begin{eqnarray*}
&&\int\nolimits_{0}^{1}\left( t-t^{n-\alpha +1}\right) \left\vert
1-2t\right\vert ^{n-\alpha }\ \left\vert f\right\vert \ (y)\ dt \\
&=&\int\nolimits_{0}^{\frac{1}{2}}\left( t-t^{n-\alpha +1}\right) \left(
1-2t\right) ^{n-\alpha }dt+\int\nolimits_{\frac{1}{2}}^{1}\left(
t-t^{n-\alpha +1}\right) \left( 2t-1\right) ^{n-\alpha }dt
\end{eqnarray*}%
\begin{equation*}
=\frac{1}{4}\left( \frac{1}{n-\alpha +2}-\frac{1}{n-\alpha +1}\right) -\frac{%
1}{2^{n-\alpha +2}}\beta \left( n-\alpha +1,n-\alpha +2\right)
\end{equation*}%
\begin{equation*}
+\frac{1}{4}\left( \frac{1}{n-\alpha +2}+\frac{1}{n-\alpha +1}\right)
-\left( -1\right) ^{n-\alpha +1}\beta \left( n-\alpha +1,n-\alpha +2\right)
\end{equation*}%
\begin{equation*}
\int\nolimits_{0}^{1}u^{n-\alpha }\left( u+1\right) ^{n-\alpha }du=\left(
-1\right) ^{n-\alpha +1}\beta \left( n-\alpha +1,n-\alpha +2\right) ,\ 
\end{equation*}

Here we used the fact that 
\begin{eqnarray*}
&&\ f(tx+(1-t)y) \\
&\leq &t^{n-\alpha }f(x)+\left\vert \left( t^{n-\alpha }-\left( 1-t\right)
^{n-\alpha }\right) \right\vert f(y)\  \\
&\leq &t^{n-\alpha }f(x)+\left\vert 2t-1\right\vert ^{n-\alpha }f(y).
\end{eqnarray*}

\begin{theorem}
Let the conditions of Theorem 3 be satisfied with $\left\vert
f^{n}\right\vert \ $is $\left( n-\alpha \right) \ $is convex. In this case,
the following inequality is true.
\end{theorem}

\begin{equation*}
\left[ \frac{\left( -1\right) ^{n}}{\left( x-y\right) ^{2}}\Gamma \left(
2\right) -\frac{\left( -1\right) ^{n}\ }{\left( x-y\right) ^{\left( n-\alpha
\right) }}\Gamma \left( n-\alpha +2\right) \right] \left(
C_{D_{x^{-}}}^{\alpha }f\right) \left( u\right)
\end{equation*}%
\begin{eqnarray*}
&\leq &\left\vert f^{\left( n\right) }\left( y\right) \right\vert \left(
\beta \left( \left( n-\alpha \right) +1,\left( n-\alpha \right) +2\right)
+\beta \left( 2,2\left( n-\alpha \right) +1\right) \right) \\
&&+\left\vert f^{\left( n\right) }\left( x\right) \right\vert \beta \left(
2,2\left( n-\alpha \right) +1\right)
\end{eqnarray*}

\begin{proof}
From Lemma 2 and the the properties of Modulus%
\begin{eqnarray*}
&&\left\vert \frac{\left( -1\right) ^{n}}{\left( x-y\right) ^{n-\alpha +1}}%
\left[ \Gamma \left( n-\alpha +1\right) -\frac{1}{x-y}\Gamma \left( n-\alpha
+1\right) \right] \left( C_{D_{x^{-}}^{\alpha }}f\right) \left( u\right)
\right\vert \\
&=&\left\vert \int\nolimits_{0}^{1}t^{n-\alpha }\left( 1-t\right) f^{\left(
n\right) }\left( tx+\left( 1-t\right) y\right) dt\right\vert
\end{eqnarray*}%
\begin{equation*}
\leq \int\nolimits_{0}^{1}t^{n-\alpha }\left( 1-t\right) \left\vert
f^{\left( n\right) }\left( tx+\left( 1-t\right) y\right) \right\vert dt
\end{equation*}%
\begin{eqnarray*}
&\leq &\left\vert f^{\left( n\right) }\left( x\right) \right\vert
\int\nolimits_{0}^{1}t^{2\left( n-\alpha \right) }\left( 1-t\right) dt \\
&&+\left\vert f^{\left( n\right) }\left( y\right) \right\vert
\int\nolimits_{0}^{1}t^{n-\alpha }\left( 1-t\right) \left\{ t^{n-\alpha
}+\left( 1-t\right) ^{n-\alpha }\right\} dt
\end{eqnarray*}%
\begin{eqnarray*}
&=&\left\vert f^{\left( n\right) }\left( x\right) \right\vert \beta \left(
2,2\left( n-\alpha \right) +1\right) \\
&&+\left\vert f^{\left( n\right) }\left( y\right) \right\vert \left\{
\int\nolimits_{0}^{1}t^{2\left( n-\alpha \right) }\left( 1-t\right)
dt+\int\nolimits_{0}^{1}t^{\left( n-\alpha \right) }\left( 1-t\right)
^{\left( n-\alpha +1\right) }dt\right\}
\end{eqnarray*}%
\begin{eqnarray*}
&=&\left\vert f^{\left( n\right) }\left( x\right) \right\vert \beta \left(
2,2\left( n-\alpha \right) +1\right) \\
&&+\left\vert f^{\left( n\right) }\left( y\right) \right\vert \left( \beta
\left( 2,2\left( n-\alpha \right) +1\right) +\beta \left( \left( n-\alpha
\right) +1,\left( n-\alpha \right) +2\right) \right)
\end{eqnarray*}
\end{proof}

\begin{theorem}
\ \ \ \ $\ \ $For all $n-\alpha \in \left( -1,1\right] \ \ $we obtain a
general elementary inequality%
\begin{equation*}
\frac{1}{2}\leq \frac{n-\alpha +3}{\left( n-\alpha +1\right) \left( n-\alpha
+2\right) }
\end{equation*}
\end{theorem}

Here $0<n\leq 1,\ 0\leq \alpha <1$

\begin{proof}
Since $t\leq t^{n-\alpha }\ \Longrightarrow \ 1-t^{n-\alpha }\leq 1-t\ $we
can write inequality, for $n-\alpha \in \left( -1,1\right] \ $ve\ $t\in %
\left[ 0,1\right] $.%
\begin{equation*}
t\left( 1-t^{n-\alpha }\right) \leq t^{n-\alpha }\left( 1-t\right) .
\end{equation*}%
Now, if we only integrate both sides of the inequality $t\left(
1-t^{n-\alpha }\right) \leq t^{n-\alpha }\left( 1-t\right) $ over $\left[ 0,1%
\right] $, \ We obtain the desired inequality: 
\begin{eqnarray*}
\int\nolimits_{0}^{1}t\left( 1-t^{n-\alpha }\right) dt &\leq
&\int\nolimits_{0}^{1}t^{n-\alpha }\left( 1-t\right) dt \\
\frac{t^{2}}{2}-\frac{t^{n-\alpha +1}}{n-\alpha +1}_{0}^{1} &\leq &\frac{%
t^{n-\alpha +1}}{^{n-\alpha +1}}-\frac{t^{n-\alpha +2}}{n-\alpha +2}_{0}^{1}
\\
\frac{1}{2} &\leq &\frac{n-\alpha +3}{\left( n-\alpha +1\right) \left(
n-\alpha +2\right) }
\end{eqnarray*}
\end{proof}

\section{Conclusion}

In fact, higher mathematics does not go beyond the study of numbers or
number systems in one way or another (e.g. derivative or integral, etc.).
One of the three basic properties of real numbers is the order relation is
especially important in optimization theory and nonlinear programming. In
this context, the convex functions discussed play a leading role. For
example, the H-H integral inequality \ \ \ and its variants, well known in
the literature, have been obtained through convex functions and classical
inequalities. The results obtained above in such that $\left( n-\alpha
\right) $ -convex and $f\in AC^{1}L[y,x]$ $\ $with $y<x$ are just a few of
the infinite number of upper bounds. Researchers interested in the subject
can approach the smallest of the upper limits by writing more specific
parameters.

\end{document}